\documentclass{article}%
\usepackage{amsmath}
\usepackage{amsfonts}
\usepackage{amsmath}
\usepackage{amssymb}
\usepackage{graphicx}%
\newtheorem{theorem}{Theorem}[section]

\newtheorem{corollary}[theorem]{Corollary}

\newtheorem{lemma}[theorem]{Lemma}

\newtheorem{proposition}[theorem]{Proposition}
\newtheorem{remark}[theorem]{Remark}

\numberwithin{equation}{section}
\newenvironment{proof}[1][Proof]{\noindent \textbf{#1.} }{\  \rule{0.5em}{0.5em}}
\begin{document}

\title{Smooth maps with singularities of bounded $\mathcal{K}$-codimensions
}
\author{Yoshifumi ANDO \thanks{This research was partially supported by Grand-in-Aid
for Scientific Research (No. 16540072).}}
\date{}
\maketitle

\begin{abstract}
Let $N$ and $P$ be smooth manifolds of dimensions $n$ and $p$ respectively
such that $n \geqq p\geqq2$ or $n<p$. Let $\mathcal{O}_{\ell}(N,P)$ denote a
$\mathcal{K}$-invarinat open subspace of $J^{\infty}(N,P)$ which consists of
all regular jets and singular jets $z$ with codim$\mathcal{K}z \leqq\ell$
(including fold jets if $n \geqq p$). An $\mathcal{O}_{\ell}$-regular map
$f:N\rightarrow P$ refers to a smooth map such that $j^{\infty}f(N)\subset
\mathcal{O}_{\ell}(N,P)$. We will prove that a continuous section $s$ of
$\mathcal{O}_{\ell}(N,P)$ over $N$ has an $\mathcal{O}_{\ell}$-regular map $f$
such that $s$ and $j^{\infty}f$ are homotopic as sections. We next study the
filtration of the group of homotopy self-equivalences of a manifold $P$ which
is constructed by the sets of $\mathcal{O}_{\ell}$-regular homotopy
self-equivalences for nonnegative integers $\ell$.

\end{abstract}

\section{Introduction}

Let $N$ and $P$ be smooth ($C^{\infty}$) manifolds of dimensions $n$ and $p$
respectively. Let $J^{k}(N,P)$ denote the $k$-jet space of the manifolds $N$
and $P$ with the projections $\pi_{N}^{k}$ and $\pi_{P}^{k}$ onto $N$ and $P$
mapping a jet onto its source and target respectively. The canonical fiber is
the $k$-jet space $J^{k}(n,p)$ of $C^{\infty}$-map germs $(\mathbb{R}%
^{n},0)\rightarrow(\mathbb{R}^{p},0)$. Let $\mathcal{K}$ denote the contact
group defined in [MaIII]. Let $\mathcal{O}(n,p)$ denote a $\mathcal{K}%
$-invariant nonempty open subset of $J^{k}(n,p)$ and let $\mathcal{O}(N,P)$
denote an open subbundle of $J^{k}(N,P)$\ associated to $\mathcal{O}(n,p)$. In
this paper a smooth map $f:N\rightarrow P$ is called an $\mathcal{O}%
$-\textit{regular map} if $j^{k}f(N)\subset\mathcal{O}(N,P)$.

We will study what is called the homotopy principle for $\mathcal{O}$-regular
maps. As for the long history of the several types of homotopy principles and
their applications we refer to the Smale-Hirsch Immersion Theorem ([Sm] and
[H]), the Feit $k$-mersion Theorem ([F]), the Phillips Submersion Theorem
([P]) and the general theorems due to Gromov ([G1]) and du Plessis ([duP1],
[duP2] and [duP3]). Furthermore, we should refer to the homotopy principle on
the $1$-jet level for fold-maps due to \`{E}lia\v{s}berg ([E1] and [E2]) (see
further references in [G2]).

Let $C_{\mathcal{O}}^{\infty}(N,P)$ denote the space consisting of all
$\mathcal{O}$-regular maps, $N\rightarrow P$ equipped with the $C^{\infty}%
$-topology. Let $\Gamma_{\mathcal{O}}(N,P)$ denote the space consisting of all
continuous sections of the fiber bundle $\pi_{N}^{k}|\mathcal{O}%
(N,P):\mathcal{O}(N,P)\rightarrow N$ equipped with the compact-open topology.
Then there exists a continuous map $j_{\mathcal{O}}:C_{\mathcal{O}}^{\infty
}(N,P)\rightarrow\Gamma_{\mathcal{O}}(N,P)$ defined by $j_{\mathcal{O}%
}(f)=j^{k}f$. If the following property (h-P) holds, then we say in this paper
that the \emph{relative homotopy principle on the existence level} holds for
$\mathcal{O}$-regular maps.

(h-P) Let $C$ be a closed subset of $N$ with $\partial N=\varnothing$. Let $s$
be a section in $\Gamma_{\mathcal{O}}(N,P)$ which has an $\mathcal{O}$-regular
map $g$ defined on a neighborhood of $C$ to $P$, where $j^{k}g=s$. Then there
exists an $\mathcal{O}$-regular map $f:N\rightarrow P$ such that $s$ and
$j^{k}f$\ are homotopic relative to a neighborhood of $C$\ by a homotopy
$s_{\lambda}$ in $\Gamma_{\mathcal{O}}(N,P)$ with $s_{0}=s$ and $s_{1}=j^{k}f$.

As important applications of [An7, Theorem 0.1] we will prove the following
relative homotopy principles in (h-P). Here, $\Sigma^{n-p+1,0}(n,p)$ refers to
\ the space consisting of all fold jets in $J^{k}(n,p)$.

\begin{theorem}
Let $n$ and $p$ be positive integers with $n\geqq p\geqq2$ or $n<p$. Let $k$
be a positive integer with $k\geqq n-|n-p|+2$. Let $\mathcal{O}(n,p)$ denote a
$\mathcal{K}$-invariant open subspace of $J^{k}(n,p)$ containing all regular
jets such that if $n\geqq p\geqq2$, then $\mathcal{O}(n,p)$ contains
$\Sigma^{n-p+1,0}(n,p)$ at least. Let $N$ and $P$ be connected smooth
manifolds of dimensions $n$ and $p$ respectively with $\partial N=\emptyset$.
Let $C$ be a closed subset of $N$. Let $s$ be a section in $\Gamma
_{\mathcal{O}}(N,P)$ which has an $\mathcal{O}$-regular map $g$ defined on a
neighborhood of $C$ to $P$, where $j^{k}g=s$.

Then there exists an $\mathcal{O}$-regular map $f:N\rightarrow P$ such that
$j^{k}f$ is homotopic to $s$ relative to a neighborhood of $C$ as sections in
$\Gamma_{\mathcal{O}}(N,P)$.
\end{theorem}

Let $\rho$\ be an integer with $\rho\geqq1$. Let $W_{\rho}^{k}$ denote the
subset consisting of all $z\in J^{k}(n,p)$ such that the codimension of
$\mathcal{K}z$ in $J^{k}(n,p)$ is not less than $\rho$ ($k$ may be $\infty
$).\ Let $\mathcal{O}_{\ell}^{k}(n,p)$ denote a $\mathcal{K}$-invariant
nonempty open subset of $J^{k}(n,p)\backslash W_{\ell+1}^{k}$. By applying
Theorem 1.1 we will prove the following theorem.

\begin{theorem}
Let $\ell$\ be a positive integer. Let $k\geqq\max\{\ell+1,n-|n-p|+2\}$ or
$k=\infty$. Let $\mathcal{O}_{\ell}^{k}(n,p)$ denote a $\mathcal{K}$-invariant
open subspace of $J^{k}(n,p)$ containing all regular jets such that if $n\geqq
p\geqq2$, then $\mathcal{O}_{\ell}^{k}(n,p)$ contains $\Sigma^{n-p+1,0}(n,p)$
at least. Then the relative homotopy principle in (h-P) holds for
$\mathcal{O}_{\ell}^{k}$-regular maps.
\end{theorem}

It is well known that any smooth map $f:N\rightarrow P$ is homotopic to a
smooth map $g:N\rightarrow P$ such that $j_{x}^{\infty}g$ is of finite
$\mathcal{K}$-codimension for any $x\in N$ (see, for example, [W, Theorem 5.1]).

There have been described many important applications of the homotopy
principles in [G2]. We only refer to the recent applications of the relative
homotopy principle on the existence level to the problems in topology such as
the elimination of singularities and the existence of $\mathcal{O}_{l}^{k}%
$-regular maps in [An1-7] and [Sa] and the relation between the stable
homotopy groups of spheres and higher singularities in [An4].

Let $P$\ be a closed manifold of dimension $p$. Let $\mathfrak{h}(P)$\ denote
the group of all homotopy classes of homotopy equivalences of $P$. Let
$\mathfrak{h}_{\ell}(P)$\ denote the subset of $\mathfrak{h}(P)$\ which
consists of all homotopy classes of maps which are homotopic\ to
$\mathcal{O}_{l}^{\infty}$-regular homotopy equivalences. In particular,
$\mathfrak{h}_{0}(P)$\ is the subset of all homotopy classes of maps\ which
are homotopic\ to diffeomorphisms of $P$.\ In this paper we will prove that
the following\ filtration%
\begin{equation}
\mathfrak{h}_{0}(P)\subset\mathfrak{h}_{1}(P)\subset\cdots\subset
\mathfrak{h}_{\ell}(P)\subset\cdots\subset\mathfrak{h}(P).
\end{equation}
is never trivial in general.

\begin{theorem}
For a given positive integer $d$, there exists a closed oriented $p$-manifold
$P$ and a sequence of positive integers $\ell_{1}$, $\ell_{2},\cdots,\ell_{d}$
with $\ell_{j}<\ell_{j+1}$ for $1\leq j<d$ such that%
\[
\mathfrak{h}_{0}(P)\varsubsetneqq\mathfrak{h}_{\ell_{1}}(P)\varsubsetneqq
\mathfrak{h}_{\ell_{2}}(P)\varsubsetneqq\cdots\varsubsetneqq\mathfrak{h}%
_{\ell_{d}}(P)\varsubsetneqq\mathfrak{h}(P).
\]

\end{theorem}

In Section 2\ we will review the results on the Boardman manifolds and the
fundamental properties of $\mathcal{K}$-equivalence and $\mathcal{K}%
$-determinacy which are necessary in this paper. In Section 3 we will recall
[An7, Theorem 0.1] and apply it in the proofs of Theorems 1.1 and 1.2. In
Section 4 we will study the nonexistence problem of $\mathcal{O}_{l}^{k}%
$-regular maps. In Section 5 we will study the filtration in (1.1) and prove
Theorem 1.3.

\section{Boardman manifolds and $\mathcal{K}$-orbits}

Throughout the paper all manifolds are Hausdorff, paracompact and smooth of
class $C^{\infty}$. Maps are basically smooth (of class $C^{\infty}$) unless
otherwise stated.

For a Boardman symbol (simply symbol) $I=(i_{1},\cdots,i_{k})$ with $n\geqq
i_{1}\geqq\cdots\geqq i_{k}\geqq0$, let $\Sigma^{I}(n,p)$ denote the Boardman
manifold of symbol $I$\ in $J^{k}(n,p)$\ which has been defined in [T], [L],
[Bo] and [MaTB]. Let $A_{n}=\mathbb{R}[[x_{1},\cdots,x_{n}]]$ denote the
formal power series of algebra on variables $x_{1},\cdots,x_{n}$. Let
$\mathfrak{m}_{n}$ be its maximal ideal and $A_{n}(k)=A_{n}/\mathfrak{m}%
_{n}^{k+1}$.\ Let $z=j_{0}^{k}f\in J^{k}(n,p)$\ where $f=(f^{1},\cdots
,f^{p}):(\mathbb{R}^{n},0)\rightarrow(\mathbb{R}^{p},0)$. We define
$\mathcal{I}(z)$\ to be the ideal in $A_{n}(k)$\ generated by the image in
$A_{n}(k)$\ of the Taylor expansions of $f^{1},\cdots,f^{p}$. It has been
proved in [Bo]\ and [MaTB] that the Boardman symbol $I(z)$\ of $z$ depends
only on the ideal $\mathcal{I}(z)$ by the notion of the Jacobian extension.
Let $\Sigma^{I}(N,P)$ denote the subbundle of $J^{k}(N,P)$\ over $N\times
P$\ associated to $\Sigma^{I}(n,p)$. Let $\Sigma_{x,y}^{I}(N,P)$ denote the
fiber of $\Sigma^{I}(N,P)$ over $(x,y)\in N\times P$.

Since codim$\Sigma^{i_{1}}(n,p)=(p-n+i_{1})i_{1}$, the following proposition
follows from [An6, Remark 2.1], which has been proved by using the results in
[Bo, Section 6].

\begin{proposition}
Let $I=(i_{1},\cdots,i_{\ell})$ be a symbol such that $i_{1}\geqq
\max\{n-p+1,1\}$ and $\Sigma^{I}(n,p)$ is nonempty. Then we have%
\[
\mathrm{codim}\Sigma^{I}(n,p)\geqq(p-n+i_{1})i_{1}+(1/2)\Sigma_{j=2}^{\ell
}\text{ }i_{j}(i_{j}+1).
\]
In particular, if $i_{\ell}>0$, then we have $\mathrm{codim}\Sigma
^{I}(n,p)\geqq|n-p|+\ell$.
\end{proposition}

Let $\Omega^{I}(n,p)$ denote the union of all Boardman manifolds $\Sigma
^{J}(N,P)$ with $J\leq I$ in the lexicographic order. We have the following
lemma (see [duP1]).

\begin{lemma}
The space $\Omega^{I}(n,p)$ is open in $J^{k}(n,p)$.
\end{lemma}

Let us review the $\mathcal{K}$-equivalence of two smooth map germs
$f,g:(N,x)\rightarrow(P,y)$, which has been introduced in [MaIII, (2.6)], by
following [Mart, II, 1]. We say\ that the above two map germs $f$\ and
$g$\ are $\mathcal{K}$-equivalent if there exists a smooth map germ
$\phi:(N,x)\rightarrow GL(\mathbb{R}^{p})$\ and a local diffeomorphism
$h:(N,x)\rightarrow(N,x)$\ such that $f(x)=\phi(x)g(h(x))$. It is known that
this $\mathcal{K}$-equivalence is nothing but the contact equivalence
introduced in [MaIII]. The contact group $\mathcal{K}$ is defined as a certain
subgroup of the group of germs of local diffeomorphisms $(N,x)\times(P,y)$ and
acts on $J_{x,y}^{k}(N,P)$. For a $k$-jet $z$ in $J_{x,y}^{k}(N,P)$ let
$\mathcal{K}z$ denote the orbit of $\mathcal{K}$ through $z$. As is well
known, $\mathcal{K}z$ is an orbit of a Lie group. Hence, $\mathcal{K}z$ is a
submanifold of $J_{x,y}^{k}(N,P)$. This fact is also observed from the above
definition. The following lemma is important in this paper.

\begin{lemma}
The Boardman manifold $\Sigma_{x,y}^{I}(N,P)$ in $J_{x,y}^{k}(N,P)$\ is
invariant with respect to the action of $\mathcal{K}$.
\end{lemma}

\begin{proof}
Let $z=j_{x}^{k}f$\ and $w=j_{x}^{k}g$\ be $k$-jets in $J_{x,y}^{k}(N,P)$ such
that two map germs $f$\ and $g$\ are $\mathcal{K}$-equivalent as above. Let
$h_{\ast}:C_{x}\rightarrow C_{x}$ be the isomorphism defined by $h_{\ast}%
(\phi)=\phi\circ h$. By the definition of $\mathcal{K}$-equivalence we have
$h_{\ast}(\mathfrak{I}(g))=\mathfrak{I}(f)$. The Thom-Boardman symbols of
$j_{x}^{k}f$ and $j_{x}^{k}g$ are determined by $\mathfrak{I}(f)$ and
$\mathfrak{I}(g)$, and are the same by [MaTB, 2, Corollary]. This proves the assertion.
\end{proof}

Let us review the results in [MaIII], [MaIV] and [MaV] which are necessary in
this paper. Let $C^{\infty}(N,x)$\ and $C^{\infty}(P,y)$\ denote the rings of
smooth function germs on $(N,x)$ and $(P,y)$\ respectively. Let $\mathfrak{m}%
_{x}$ and $\mathfrak{m}_{y}$\ denote their maximal ideals respectively. Let
$f:(N,x)\rightarrow(P,y)$\ be a germ of a smooth map. Let $f^{\ast}:C^{\infty
}(P,y)\rightarrow C^{\infty}(N,x)$\ denote the homomorphism defined
by\ $f^{\ast}(a)=a\circ f$. Let $\theta(N)_{x}$\ denote the $C^{\infty}%
(N,x)$-module of all germs at $x$ of smooth vector fields on $(N,x).$ We
define $\theta(P)_{y}$ similarly for $y\in P$.\ Let $\theta(f)_{x}$\ denote
the $C^{\infty}(N,x)$-module of germs at $x$ of smooth vector fields along
$f$, namely which consists of all smooth germs $\varsigma:(N,x)\rightarrow
TP$\ such that $p_{P}\circ\varsigma=f$. Here, $p_{P}:TP\rightarrow P$\ is the
canonical projection. Then we have the homomorphisms%
\begin{equation}
tf:\theta(N)_{x}\rightarrow\theta(f)_{x}%
\end{equation}
defined by $tf(u_{N})=df\circ u_{N}$\ for $u_{N}\in\theta(N)_{x}$. For a
singular jet $z=j_{0}^{k}f\in J^{k}(N,P)$ there has been defined the
isomorphism%
\begin{equation}
T_{z}(J_{x,y}^{k}(N,P))\longrightarrow\mathfrak{m}_{x}\theta(f)_{x}%
/\mathfrak{m}_{x}^{k+1}\theta(f)_{x}%
\end{equation}
in [MaIII, (7.3)] such that $T_{z}(\mathcal{K}z)$ corresponds to
$tf(\mathfrak{m}_{x}\theta(N)_{x})+f^{\ast}(\mathfrak{m}_{y})(\theta(f)_{x})$
modulo $\mathfrak{m}_{x}^{k+1}\theta(f)_{x}$. We do not here explain the
definition. According to [MaIII] we define $d(f,\mathcal{K})$ to be%
\begin{equation}
\dim\mathfrak{m}_{x}\theta(f)_{x}/(tf(\mathfrak{m}_{x}\theta(N)_{x})+f^{\ast
}(\mathfrak{m}_{y})(\theta(f)_{x})),
\end{equation}
which is equal to codim$\mathcal{K}z$.

\section{Proofs of Theorems 1.1 and 1.2.}

In this section we prove Theorems 1.1 and 1.2.

Let $k$ be a positive integer. Let $W_{\rho}^{k}=W_{\rho}^{k}(n,p)$ denote the
subset consisting of all $z\in J^{k}(n,p)$\ such that the codimension of
$\mathcal{K}z$ in $J^{k}(n,p)$ is not less than $\rho$. The following
lemma\ has been observed in [MaV, Section 7 and Proof of Theorem 8.1].

\begin{lemma}
Let $\rho$ be an integer with $\rho\geqq1$. Then $W_{\rho}^{k}$ is an
algebraic subset of $J^{k}(n,p)$.
\end{lemma}

The order of $\mathcal{K}$-determinacy is estimated by the codimension of a
$\mathcal{K}$-orbit as follows.

\begin{proposition}
Let $k$ be an integer with $k>\rho$. Let $z=j^{k}f$\ be a singular jet in
$J^{k}(n,p)\backslash W_{\rho+1}^{k}$. Then $z$\ is $\mathcal{K}$-$k$-determined.
\end{proposition}

\begin{proof}
It follows from [W, Theorem 1.2 (iii)] that if $d=$\textrm{codim}%
$\mathcal{K}z$, then $z$ is $\mathcal{K}$-$(d+1)$-determined. Hence, if $z\in
J^{k}(n,p)\backslash W_{\rho+1}^{k}$, then $d\leq\rho$ and $z$ is
$\mathcal{K}$-$k$-determined.
\end{proof}

We define the bundle homomorphism%
\begin{align}
\mathbf{d}  &  :(\pi_{N}^{k})^{\ast}(TN)\longrightarrow(\pi_{k-1}^{k})^{\ast
}(TJ^{k-1}(N,P)),\\
\mathbf{d}_{1}  &  :(\pi_{N}^{k})^{\ast}(TN)\longrightarrow(\pi_{P}^{k}%
)^{\ast}(TP).\nonumber
\end{align}
Let $w=j_{x}^{k}f\in J_{x,y}^{k}(N,P)$ and $z=\pi_{k}^{k}(w)$. Then we have
$j^{k-1}f:(N,x)\rightarrow(J^{k-1}(N,P),z)$ and $d(j^{k-1}f):T_{x}N\rightarrow
T_{z}(J^{k-1}(N,P))$. We set%
\[
\mathbf{d}_{z}(w,\mathbf{v})=(w,d(j^{k-1}f)(\mathbf{v}))\text{ \ and
\ }(\mathbf{d}_{1})_{z}(w,\mathbf{v})=(w,df(\mathbf{v})).
\]
Let $I^{\prime}$ be a symbol of length $k$. Let $\mathbf{K}(\Sigma^{I^{\prime
}})$ denote the kernel subbundle of $(\pi_{N}^{k}|\Sigma^{I^{\prime}%
}(N,P))^{\ast}(TN)$ defined by%
\[
\mathbf{K}(\Sigma^{I^{\prime}})_{w}=(w,\mathrm{Ker}(d_{x}f)).
\]

The following theorem follows from the corresponding assertion for the case
$k=\infty$\ in [B, (7.7)]. This is very important in the proof of Theorem 1.1.

\begin{theorem}
If $I^{\prime}=(i_{1},\cdots,i_{k-2},0,0)$ and $I=(i_{1},\cdots,i_{k-2,}0)$,
then we have%
\[
\mathbf{d}(\mathbf{K}(\Sigma^{I^{\prime}})_{w})\cap(\pi_{k-1}^{k}%
|\Sigma^{I^{\prime}}(N,P))^{\ast}(T(\Sigma^{I}(N,P))_{w}=\{0\}
\]
for any $w\in\Sigma^{I^{\prime}}(N,P)$.
\end{theorem}

Let us review a general condition on $\mathcal{O}(n,p)$\ for the relative
homotopy principle on the existence level in [An7]. We say that a nonempty
$\mathcal{K}$-invariant open subset $\mathcal{O}(n,p)$ is \emph{admissible} if
$\mathcal{O}(n,p)$ consists of all regular jets and a finite number of
disjoint $\mathcal{K}$-invariant nonempty submanifolds $V^{i}(n,p)$ of
codimension $\rho_{i}$ ($1\leq i\leq\iota$) such that the following properties
(H-i) to (H-v) are satisfied.

(H-i) $V^{i}(n,p)$ consists of singular $k$-jets of rank $r_{i}$, namely,
$V^{i}(n,p)\subset\Sigma^{n-r_{i}}(n,p)$.

(H-ii) For each $i$, the set $\mathcal{O}(n,p)\backslash\{ \cup_{j=i}^{\iota
}V^{j}(n,p)\}$ is an open subset.

(H-iii) For each $i$ with $\rho_{i}\leq n$, there exists a $\mathcal{K}%
$-invariant submanifold $V^{i}(n,p)^{(k-1)}$\ of $J^{k-1}(n,p)$\ such that
$V^{i}(n,p)$\ is open in $(\pi_{k-1}^{k})^{-1}(V^{i}(n,p)^{(k-1)})$.

(H-iv) If $n\geqq p\geqq2$, then $V^{1}(n,p)=\Sigma^{n-p+1,0}(n,p)$.

\noindent Here, $\Sigma^{n-p+1,0}(n,p)$ denotes the Thom-Boardman manifold in
$J^{k}(n,p)$, which consists of $\mathcal{K}$-orbits of fold jets. Let
$V^{i}(N,P)$ denote the subbundle of $J^{k}(N,P)$\ associated to $V^{i}(n,p)$.
Let $\mathbf{K}(V^{i})$ be the kernel bundle in $(\pi_{N}^{k})^{\ast
}(TN)|_{V^{i}(N,P)}$\ defined by $\mathbf{K}(V^{i})_{z}=(z,\mathrm{Ker}%
(d_{x}f))$.

(H-v) For each $i$ with $\rho_{i}\leq n$ and any $z\in V^{i}(N,P)$, we have%
\begin{equation}
\mathbf{d}(\mathbf{K}(V^{i})_{z})\cap(\pi_{k-1}^{k}|V^{i}(N,P))^{\ast}%
(T(V^{i}(N,P)^{(k-1)})_{z}=\{0\}.
\end{equation}

Then we have proved the following theorem in [An7, Theorem 0.1].

\begin{theorem}
Let $k\geqq n-|n-p|+2$. Let $n\geqq p\geqq2$ or $n<p$. Let $\mathcal{O}(n,p)$
denote an admissible open subspace of $J^{k}(n,p)$. Then the relative homotopy
principle in (h-P) holds for $\mathcal{O}$-regular maps.
\end{theorem}

We set%
\[
V_{I}(n,p)=\mathcal{O}(n,p)\cap\Sigma^{I}(n,p).
\]
Let $J=(j_{1},\cdots,j_{k})$ be a symbol of a singular jet with codim$\Sigma
^{J}(n,p)\leq n$. If $k\geqq n-|n-p|+2$, we have by Proposition 2.1 that
$i_{k-1}=i_{k}=0$. Indeed, if $i_{k-1}>0$, then%
\[
\text{codim}\Sigma^{J}(n,p)\geqq|n-p|+k-1\geqq n+1.
\]
So we set $J=(j_{1},\cdots,j_{k-2},0,0)$, $J^{\ast}=(i_{1},\cdots,i_{k-2,}0)$
and%
\[
V_{J^{\ast}}(n,p)^{(k-1)}=\pi_{k-1}^{k}(\mathcal{O}(n,p))\cap\Sigma^{J\ast
}(n,p).
\]

\begin{lemma}
Let $J=(j_{1},\cdots,j_{k-2},0,0)$ and $J^{\ast}=(j_{1},\cdots,j_{k-2,}0)$ be
as above. Then $V_{J}(n,p)$ is open in $(\pi_{k-1}^{k})^{-1}(V_{J^{\ast}%
}(n,p)^{(k-1)})$.
\end{lemma}

\begin{proof}
It is evident that%
\[
\Sigma^{J}(n,p)=(\pi_{k-1}^{k})^{-1}(\Sigma^{J^{\ast}}(n,p))\text{ \ and
\ }\mathcal{O}(n,p)\subset(\pi_{k-1}^{k})^{-1}(\pi_{k-1}^{k}(\mathcal{O}%
(n,p))).
\]
So we have $V_{J}(n,p)\subset(\pi_{k-1}^{k})^{-1}(V_{J^{\ast}}(n,p)^{(k-1)})$.
Since $\pi_{k-1}^{k}$ is an open map, we have that $V_{J}(n,p)$ is an open
subset of $(\pi_{k-1}^{k})^{-1}(V_{J^{\ast}}(n,p)^{(k-1)})$.
\end{proof}

Let us prove Theorem 1.1.

\begin{proof}
[Proof of Theorem 1.1]By Theorem 3.4 it is enough to prove that $\mathcal{O}%
(n,p)$ is admissible. Let $J$\ be a symbol of length $k$. By Lemma 2.3,
$V_{J}(n,p)$ is $\mathcal{K}$-invariant. We have that

(H1) $\mathcal{O}(n,p)$\ is decomposed into a finite union of all $V_{J}(n,p)$,

(H2) For each symbol $J$, the set $\mathcal{O}(n,p)\cap\Omega^{J}(n,p)$\ is an
open subset of $\mathcal{O}(n,p)$,

(H3) $V_{J}(n,p)$ is open in $(\pi_{k-1}^{k})^{-1}(V_{J^{\ast}}(n,p)^{(k-1)})$
by lemma 3.5,

(H4) If $n\geqq p\geqq2$, then $\mathcal{O}(n,p)\supset\Sigma^{n-p+1,0}(n,p)$
by the assumption,

(H5) Property (3.2) holds for $V_{J}(n,p)$ by Theorem 3.3 and Lemma 3.5.

Since $\mathcal{O}(n,p)$ satisfies the properties (H1) to (H5), we have proved
Theorem 1.1.
\end{proof}

We next prove Theorem 1.2.

\begin{proof}
[Proof of Theorem 1.2]If $\ell$ is finite, then it follows from Lemma 3.2 that
if $k>\ell$, then any $k$-jet $z$ of $J^{k}(n,p)\backslash W_{\ell+1}^{k}$ is
$\mathcal{K}$-$k$-determined and we have%
\[
(\pi_{k}^{\infty})^{-1}(\mathcal{O}_{\ell}^{k}(n,p))=\mathcal{O}_{\ell
}^{\infty}(n,p).
\]
Therefore, if $k\geqq\max\{\ell+1,n-|n-p|+2\}$, then the relative homotopy
principle in (h-P) holds for $\mathcal{O}_{\ell}^{k}$-regular maps by Theorem
1.1 and also for $\mathcal{O}_{\ell}^{\infty}$-regular maps.
\end{proof}

\begin{corollary}
Under the same assumption of Theorem 1.2, given a map $f:N\rightarrow P$ is
homotopic to an $\mathcal{O}_{\ell}^{k}$-regular map if and only if there
exists a section $s\in\Gamma_{\mathcal{O}_{\ell}^{k}}(N,P)$ such that $\pi
_{P}^{k}\circ s$ is homotopic to $f$.
\end{corollary}

\begin{corollary}
Let $\mathfrak{h}_{\ell}(P)$ be as in Introduction. Then the homotopy class of
a homotopy equivalence $f:P\rightarrow P$ lies in $\mathfrak{h}_{\ell}(P)$\ if
and only if $j^{\infty}f$ is homotopic to a section in $\Gamma_{\mathcal{O}%
_{\ell}^{\infty}}(N,P)$.
\end{corollary}

Here we give two remarks.

\begin{remark}
\textrm{Let }$W_{\infty}^{\infty}$\textrm{ denote the subspace of }$J^{\infty
}(n,p)$\textrm{ which consists of all jets }$z$\textrm{ such that any smooth
map germ }$f$\textrm{ with }$z=j^{\infty}f$\textrm{ is not finitely
determined. Let }$W_{\infty}^{\infty}(N,P)$\textrm{ is the subbundle of
}$J^{\infty}(N,P)$\textrm{ associated to }$W_{\infty}^{\infty}$\textrm{. It
has been proved (see, for example, [W, Theorem\ 5.1]) that }$W_{\infty
}^{\infty}$\textrm{ is not of finite codimension in }$J^{\infty}%
(n,p)$\textrm{. Consequently, the space of all smooth maps }$f:N\rightarrow
P$\textrm{ with }$j^{\infty}f(N)\subset J^{\infty}(N,P)\backslash W_{\infty
}^{\infty}(N,P)$\textrm{ is dense in }$C^{\infty}(N,P)$\textrm{. In other
words if }$N$\textrm{ is compact, then a smooth map }$f:N\rightarrow
P$\textrm{ has an integer }$\ell$\textrm{ such that }$f$\textrm{ is homotopic
to an }$O_{\ell}^{\infty}$\textrm{-regular map.}
\end{remark}

\begin{remark}
\textrm{It is very important to study the topology of the space }$W_{\ell
+1}^{k}(n,p)$\textrm{ and obstructions for finding an }$O_{\ell}^{k}%
$\textrm{-regular map. The Thom polynomials related to }$W_{\ell+1}^{k}%
(n,p)$\textrm{ have been studied in the dimensions }$n=p\leqq8$\textrm{ in [O]
and [F-R].}
\end{remark}

\section{Nonexistence theorems}

In this section we will discuss the nonexistence of $\mathcal{O}_{\ell}^{k}%
$-regular maps $f:N\rightarrow P$. Let $W_{\ell+1}^{k}(N,P)$ denote the
subbundle of $J^{k}(N,P)$ associated to $W_{\ell+1}^{k}(n,p)$. By the homotopy
principle for $\mathcal{O}_{\ell}^{k}$-regular maps in Theorem 1.2, the
existence of a section of $J^{k}(N,P)\backslash W_{\ell+1}^{k}(N,P)$ over $N$
is equivalent to the existence of an $\mathcal{O}_{\ell}^{k}$-regular map.
However, it is not so easy to find obstructions associated to $W_{\ell+1}%
^{k}(N,P)$ such as Thom polynomials of $W_{\ell+1}^{k}(N,P)$, and so we will
adopt a method applied in [An1], [I-K] and [duP4] in this section.

For $k\geqq p+1$, let $\Sigma(n,p;k)$ denote the algebraic subset of all
$C^{\infty}$-nonstable $k$-jets of $J^{k}(n,p)$ defined in [MaV]. Note that
for $k^{\prime}>k$, $(\pi_{k}^{k^{\prime}})^{-1}(\Sigma(n,p;k))=\Sigma
(n,p;k^{\prime})$. We have proved the following\ proposition in [An1,
Corollary 5.6].

\begin{proposition}
Let $k\geqq p+1$. If%
\[
(p-n+i)(\frac{1}{2}i(i+1)-p+n)-i^{2}\geqq n,
\]
then we have that $\Sigma^{i}(n,p)\subset\Sigma(n,p;k)$.
\end{proposition}

In [I-K] the following proposition has been proved, while it has not been
stated explicitly and\ the proof has been given in the context without the
details. So we give a sketchy proof.

\begin{proposition}
[\lbrack I-K\rbrack]Let $\ell$ be a nonnegative integer and $k\geqq p+\ell+1$.
If%
\[
(p-n+i)(\frac{1}{2}i(i+1)-p+n)-i^{2}\geqq n+\ell,
\]
then we have that $\Sigma^{i}(n,p)\subset W_{\ell+1}^{k}(n,p)$. In particular,
if $n=p$ and $\frac{1}{2}i^{2}(i-1)\geqq n+\ell,$ then we have that
$\Sigma^{i}(n,n)\subset W_{\ell+1}^{k}(n,n).$
\end{proposition}

\begin{proof}
Take a jet $z$ in $\Sigma^{i}(n,p)$ such that $z=j_{0}^{k}f$. Suppose that
$z\notin W_{\ell+1}^{k}$, and hence codim$\mathcal{K}z\leqq\ell$. By [MaIV]
there exists a versal unfolding $F:(\mathbb{R}^{n}\times\mathbb{R}^{\ell
},0)\rightarrow(\mathbb{R}^{p}\times\mathbb{R}^{\ell},0)$ of $f$ and
$j_{(0,0)}^{k}F\notin\Sigma(n+\ell,p+\ell;k)$. Here, we note that
$j_{(0,0)}^{k}F$ is of kernel rank $i$. By the assumption and Proposition 4.1
we have%
\[
\Sigma^{i}(n+\ell,p+\ell)\subset\Sigma(n+\ell,p+\ell;k).
\]
This implies $j_{(0,0)}^{k}F\in\Sigma(n+\ell,p+\ell;k)$. This is a
contradiction. Hence, $z$ lies in $W_{\ell+1}^{k}$.
\end{proof}

We show the following proposition by applying Proposition 4.2.

\begin{proposition}
Let $\ell$ be a nonnegative integer and $k\geqq p+\ell+1$. If $\Sigma
^{i}(n,p)\subset W_{\ell+1}^{k}(n,p)$, then we have that for any positive
integer $m$, $\Sigma^{i}(m+n,m+p)\subset W_{\ell+1}^{k}(m+n,m+p)$.
\end{proposition}

\begin{proof}
Let $z=j_{0}^{k}f\in\Sigma^{i}(m+n,m+p)$. Setting $\alpha=j_{0}^{1}f$, we
identify $\alpha$ with the homomorphism $\mathbb{R}^{m+n}\rightarrow
\mathbb{R}^{m+p}$.\ Let $\mathrm{Ker}(\alpha)^{\perp}$ and
$\mathrm{\operatorname{Im}}(\alpha)^{\perp}$ be the orthogonal complement of
the kernel $\mathrm{Ker}(\alpha)$ and the image $\mathrm{\operatorname{Im}%
}(\alpha)$ of $\alpha$ respectively. Let $L$ and $M$ be subspaces of
$\mathrm{Ker}(\alpha)^{\perp}$ and $\mathrm{\operatorname{Im}}(\alpha)$ of
dimension $m$ such that $\alpha$ maps $L$ onto $M$ isomorphically. Let
$L^{\perp}$ and $M^{\perp}$ be their orthogonal complements in $\mathrm{Ker}%
(\alpha)^{\perp}$ and $\mathrm{\operatorname{Im}}(\alpha)$ respectively. Then
$\alpha$\ is decomposed as in the following exact sequence.%
\[
0\rightarrow\mathrm{Ker}(\alpha)\rightarrow L\oplus L^{\perp}\oplus
\mathrm{Ker}(\alpha)\overset{\alpha}{\rightarrow}M\oplus M^{\perp}%
\oplus\mathrm{\operatorname{Im}}(\alpha)^{\perp}\rightarrow
\mathrm{\operatorname{Im}}(\alpha)^{\perp}\rightarrow0
\]
Let us choose coordinates%
\[
(u_{1},\cdots,u_{m})\text{, }(u_{m+1},\cdots,u_{m+n-i})\text{ and
}(u_{m+n-i+1},\cdots,u_{m+n})
\]
of $L$, $L^{\perp}$ and $\mathrm{Ker}(\alpha)$, and coordinates%
\[
(y_{1},\cdots,y_{m})\text{, }(y_{m+1},\cdots,y_{m+n-i})\text{ and
}(y_{m+n-i+1},\cdots,y_{m+p})\
\]
of $M$, $M^{\perp}$ and $\mathrm{\operatorname{Im}}(\alpha)^{\perp}%
$\ respectively. Since $\alpha$ maps $L$ onto $M$ isomorphically, there exist
the new coordinates $(x_{1},\cdots,x_{m+n})$\ of $\mathbb{R}^{m+n}$\ such that%
\[
x_{j}=x_{j}(u_{1},\cdots,u_{m+n})\text{ }(1\leq j\leq m)\text{ \ and }%
x_{j}=u_{j}\text{ }(m+1\leq j\leq m+n)
\]
and that%
\begin{equation}
y_{j}\circ f(x_{1},\cdots,x_{m+n})=x_{j}\text{ }(1\leq j\leq m).
\end{equation}

Setting $\overset{\bullet}{x}=(x_{m+1},\cdots,x_{m+n})$, we define the map
$g:(\mathbb{R}^{n},0)\rightarrow(\mathbb{R}^{p},0)$ by%
\[
y_{j}\circ g(\overset{\bullet}{x})=y_{j}\circ f(0,\cdots,0,\overset{\bullet
}{x})\text{ \ \ }(m+1\leq j\leq m+p).
\]
Then $f$ is an unfolding of $g$\ by (4.1) and $g$ is of kernel rank $i$ at the
origin.\ We next prove by following the argument and the notation used in
[MaIV, Section 1] that $d(g,\mathcal{K})$\ is equal to $d(f,\mathcal{K})$.
Define $\pi:\theta(f)\rightarrow\theta(g)$\ by%
\[
\pi\left(
{\displaystyle\sum\limits_{j=1}^{m}}
a_{j}tf(\frac{\partial}{\partial xj})+\sum\limits_{j=m+1}^{m+p}a_{j}%
(\frac{\partial}{\partial yj}\circ f)\right)  =%
{\displaystyle\sum\limits_{j=m+1}^{m+p}}
a_{j}^{\prime}(\frac{\partial}{\partial yj}\circ g),
\]
where $a_{j}\in C^{\infty}(\mathbb{R}^{m+n},0)$, $a_{j}^{\prime}\in C^{\infty
}(\mathbb{R}^{n},0)$\ and $a_{j}^{\prime}(\overset{\bullet}{x})=a_{j}%
(0,\cdots,0,\overset{\bullet}{x})$. We note that%
\[%
\begin{array}
[c]{l}%
tf(\partial/\partial x_{j})=(\partial/\partial y_{j})\circ f+\sum
_{t=m+1}^{m+p}(\partial y_{t}\circ f/\partial x_{j})(\partial/\partial
y_{t})\circ f\text{ \ }(1\leq j\leq m),\\
tf(\partial/\partial x_{j})=\sum_{t=m+1}^{m+p}(\partial y_{t}\circ f/\partial
x_{j})(\partial/\partial y_{t})\circ f\text{ \ }(m+1\leq j\leq m+n),\\
(\partial y_{t}\circ f/\partial x_{j})(0,\cdots,0,\overset{\bullet}%
{x})=(\partial y_{t}\circ g/\partial x_{j})(\overset{\bullet}{x})\text{
\ }(m+1\leq t\leq m+p).
\end{array}
\]
Since%
\[
y_{t}\circ f(x_{1},\cdots,x_{m+n})-y_{t}\circ f(0,\cdots,0,\overset{\bullet
}{x})=\sum_{u=1}^{m}x_{u}b_{u}(x_{1},\cdots,x_{m+n}),
\]
for some $b_{j}\in C^{\infty}(\mathbb{R}^{m+n},0)$, we have%
\[
\partial y_{t}\circ f/\partial x_{j}-\partial y_{t}\circ g/\partial x_{j}%
=\sum_{u=1}^{m}x_{u}(\partial b_{u}/\partial x_{j})\text{ \ \ }(m+1\leq j\leq
m+n).
\]
Hence, the assertion follows from an elementary calculation under the
definition in (2.3).

Since $j_{0}^{k}g\in\Sigma^{i}(n,p)\subset W_{\ell+1}^{k}(n,p)$, we have
$d(g,\mathcal{K})\geqq\ell+1$. Hence, we have $d(f,\mathcal{K})\geqq\ell+1$.
This shows $z\in W_{\ell+1}^{k}(m+n,m+p)$. This is what we want.
\end{proof}

Let $\xi$\ be a stable vector bundle over a space. Let $\mathbf{c}(\Sigma
^{i},\xi)$ denote the determinant of the $(p-n+i)$-matrix whose $(s,t)$%
-component is the $(i+s-t)$-th Stiefel-Whitney class $W_{i+s-t}(\xi)$. If
$n-p$ and $i$ are even, say $n-p=2u$ and $i=2v$, and if $\xi$\ is orientable,
then $\mathbf{c}_{\mathbb{Z}}(\Sigma^{i},\xi)$ expresses the determinant of
the $(v-u)$-matrix whose $(s,t)$-component is the $(v+s-t)$-th Pontrjagin
class $P_{v+s-t}(\xi)$.%
\[
\left\vert
\begin{array}
[c]{ccc}%
W_{i} & \cdots & W_{n-p+1}\\
\vdots & \ddots & \vdots\\
W_{p-n+2i-1} & \cdots & W_{i}%
\end{array}
\right\vert \text{ \ \ and \ \ }\left\vert
\begin{array}
[c]{ccc}%
P_{v} & \cdots & P_{u+1}\\
\vdots & \ddots & \vdots\\
P_{2v-u-1} & \cdots & P_{v}%
\end{array}
\right\vert
\]
Let $\tau_{X}$ denote the stable tangent bundle of a manifold $X$. If
$f:N\rightarrow P$ is a smooth map transverse to $\Sigma^{i}(N,P)$\ and
$\xi=\tau_{N}-f^{\ast}(\tau_{P})$, then $\mathbf{c}(\Sigma^{i},\xi)$ (resp.
$\mathbf{c}_{\mathbb{Z}}(\Sigma^{i},\xi)$) is equal to the (resp. integer)
Thom polynomial of the topological closure of $(j^{k}f)^{-1}(\Sigma^{i}(N,P))$
([Po], [Ro] and see also [An1, Proposition 5.4]). If it does not vanish, then
$(j^{k}f)^{-1}(\Sigma^{i}(N,P))$ cannot be empty by the obstruction theory in
[St]. Hence, we have the following corollary of Propositions 4.2 and 4.3.

\begin{corollary}
Let $f:M\rightarrow Q$\ be a smooth map with $\dim M=m+n$ and $\dim Q=m+p$.
Under the same assumption of Proposition 4.2. we assume that either

(i) $\mathbf{c}(\Sigma^{i},\tau_{M}-f^{\ast}(\tau_{Q}))$ does not vanish, or

(ii) $M$ and $\tau_{M}-f^{\ast}(\tau_{Q})$ are orientable, $n-p$ and $i$ are
even and$\ \mathbf{c}_{\mathbb{Z}}(\Sigma^{i},\tau_{M}-f^{\ast}(\tau_{Q}))$
does not vanish.

\noindent Then $f$ is not homotopic to any $\mathcal{O}_{\ell}^{k}$-regular map.
\end{corollary}

\section{Homotopy equivalences}

In this section we will study the filtration in (1.1) in Introduction by
applying Corollaries 3.7 and 4.4 and Remark 3.8.

Let us first review what is called the Sullivan's exact sequence in the
surgery theory following [M-M] (see also [K-M], [Su] and [Br]).

In what follows $P$\ is a closed and oriented $n$-manifold. We define the set
$\mathcal{S}(P)$\ to be the set of all equivalence classes of homotopy
equivalences $f:N\rightarrow P$ of degree $1$ under the following equivalence
relation. Let\ $N_{j}$\ be closed oriented $n$-manifolds and let $f_{j}%
:N_{j}\rightarrow P$ be homotopy equivalences of degree $1$\ ($j=1,2$). We say
that $f_{1}$\ and $f_{2}$\ are equivalent if there exists an $h$-cobordism
$W$\ of $N_{1}$\ and $N_{2}$\ and a homotopy equivalence $F:(W,N_{1}%
\cup(-N_{2}))\rightarrow(P\times\lbrack0,1],P\times0\cup(-P)\times1)$\ of
degree $1$\ such that $F|N_{j}=f_{j}$\ ($j=1,2$).

Let $O(k)$ denote the rotation group of $\mathbb{R}^{k}$\ and let $G_{k}%
$\ denote the space of all homotopy equivalence of the $(k-1)$-sphere
$S^{k-1}$\ equipped with the compact-open topology. By considering the
canonical inclusions $O(k)\rightarrow O(k+1)$\ and $G_{k}\rightarrow G_{k+1}$,
we set $O=\lim_{k\rightarrow\infty}O(k)$\ and $G=\lim_{k\rightarrow\infty
}G_{k}$. Let $BO$\ and $BG$\ denote the classifying spaces for $O$ and $G$.
Then we have the canonical maps $\pi(m):BO(m)\rightarrow BG(m)$ and
$\pi:BO\rightarrow BG$, which are regarded as fibrations with fibers
$G(m)/O(m)$\ and $G/O$ respectively. For a sufficiently large number $m$, let
$\eta_{O(m)}$ denote the universal vector bundle over $BO(m)$ and let
$i_{G/O}:G(m)/O(m)\rightarrow BO(m)$ be the inclusion of a fiber. Set
$\eta_{G/O}=(i_{G/O})^{\ast}\eta_{O(m)}$. Then $\eta_{G/O}$ has a
trivialization $t_{G/O}:\eta_{G/O}\rightarrow\mathbb{R}^{m}$ as a spherical fibration.

We next recall the surgery obstruction $\mathfrak{s}_{4q}^{P}%
:[P,G/O]\rightarrow\mathbb{Z}$ only in the case of $n=4q$. For $[\alpha
]\in\lbrack P,G/O]$ let $\eta=\alpha^{\ast}(\eta_{G/O})$ with the canonical
bundle map $\overline{\alpha}:\eta\rightarrow\eta_{G/O}$ covering $\alpha$ and
the projection $\pi_{\eta}$ onto $P$. We deform $t_{G/O}\circ\overline{\alpha
}$ to a map transverse to $0\in\mathbb{R}^{m}$ and let $M$ be the inverse
image of $0$ with a map $\pi_{\eta}|M:M\rightarrow P$ of degree $1$. We define
$\mathfrak{s}_{4q}^{P}([\alpha])=(1/8)(\sigma(M)-\sigma(P))$. If $P$ is simply
connected in addition, then there have been defined an injection
$j^{P}:\mathcal{S}(P)\rightarrow\lbrack P,G/O]$ such that if $\mathfrak{s}%
_{4q}^{P}([\alpha])=0$, $\pi_{\eta}|M$ is deformed to a homotopy equivalence
$f:N\rightarrow P$ of degree $1$\ under a certain cobordism. The following is
the Sullivan's exact sequence.%
\[
0\longrightarrow\mathcal{S}(P)\overset{j^{P}}{\longrightarrow}[P,G/O]\overset
{\mathfrak{s}_{4q}^{P}}{\longrightarrow}\mathbb{Z}%
\]

Let us recall the \textit{cobordism group} $\Omega_{n}^{h-eq}$ of homotopy
equivalences of degree $1$ in [An5]. Let $N_{j}$ and $P_{j}$\ be oriented
closed $n$-manifolds and let $f_{j}:N_{j}\rightarrow P_{j}$\ be homotopy
equivalences\ of degree $1$\ ($j=1,2$). We say that $f_{1}$ and $f_{2}$\ are
cobordant if there exists an oriented $(n+1)$-manifold $W$, $V$ and a
homotopy\ equivalence $F:(W,\partial W)\rightarrow(V,\partial V)$ of degree
$1$ such that $\partial W=N_{1}\cup(-N_{2})$, $\partial V=P_{1}\cup(-P_{2}%
)$\ and $F|N_{j}=f_{j}$. The cobordism class of $f:N\rightarrow P$ is denoted
by $[f:N\rightarrow P]$.\ Let $\Omega_{n}^{h-eq}$\ denote the set which
consists of all cobordism classes of homotopy equivalences of degree $1$. We
provide $\Omega_{n}^{h-eq}$ with a module structure by setting

$\bullet$ $[f_{1}:N_{1}\rightarrow P_{1}]+[f_{2}:N_{2}\rightarrow
P_{2}]=[f_{1}\cup f_{2}:N_{1}\cup N_{2}\rightarrow P_{1}\cup P_{2}]$,

$\bullet$ $-[f:N\rightarrow P]=[f:(-N)\rightarrow(-P)].$

\noindent The null element is defined to be $[f:N\rightarrow P]$ which bound a
homotopy\ equivalence $F:(W,\partial W)\rightarrow(V,\partial V)$ of degree
$1$ such that $\partial W=N$, $\partial V=P$\ and $F|N=f$. Even if $P$ is not
simply connected, we can find $f_{1}:N_{1}\rightarrow P_{1}$ with $P_{1}$
being simply connected in the same cobordism class by killing $\pi
_{1}(N)\approx\pi_{1}(P)$ by usual surgery.

Let $\mathbf{c}_{\mathbb{Q}}(\Sigma^{2i},\eta_{G/O})$ denote the image of
$\mathbf{c}_{\mathbb{Z}}(\Sigma^{2i},\eta_{G/O})$\ in $H^{4i^{2}%
}(G/O;\mathbb{Q})$. Let $\alpha=j^{P}([f:N\rightarrow P])$ and let
$c_{P}:P\rightarrow BSO$ be a classifying map of the tangent bundle $TP$\ of
$P$. Then it induces the homomorphism $\mathcal{C}_{2i}:{\Omega}_{4q}%
^{h-eq}\rightarrow H_{4q-4i^{2}}(G/O;\mathbb{Q})$\ defined by%
\begin{align*}
\mathcal{C}_{2i}([f  &  :N\rightarrow P])=\mathbf{c}_{\mathbb{Q}}(\Sigma
^{2i},\eta_{G/O})\cap\alpha([P])\\
&  =\mathbf{c}_{\mathbb{Q}}(\Sigma^{2i},\eta_{G/O})\otimes1\cap(\alpha\times
c_{P})_{\ast}([P]),
\end{align*}
under the identification%
\[
H_{4q-4i^{2}}(G/O;\mathbb{Q})=H_{4q-4i^{2}}(G/O;\mathbb{Q})\otimes1
\]
in $\Sigma_{j=0}^{q-i^{2}}H_{4j}(G/O;\mathbb{Q})\otimes H_{4q-4i^{2}%
-4j}(BSO;\mathbb{Q})$. We have that%
\begin{align*}
\mathcal{C}_{2i}(\alpha)  &  =\mathbf{c}_{\mathbb{Q}}(\Sigma^{2i},\eta
_{G/O})\cap(\alpha)_{\ast}([P])\\
&  =\mathbf{c}_{\mathbb{Q}}(\Sigma^{2i},\eta_{G/O})\cap(\alpha\circ f)_{\ast
}([N])\\
&  =(\alpha\circ f)_{\ast}((\alpha\circ f)^{\ast}(\mathbf{c}_{\mathbb{Q}%
}(\Sigma^{2i},\eta_{G/O}))\cap\lbrack N])\\
&  =(\alpha\circ f)_{\ast}(\mathbf{c}_{\mathbb{Z}}\mathbf{(}\Sigma^{2i}%
,\tau_{N}-f^{\ast}(\tau_{P}))\cap\lbrack N]).
\end{align*}

Furthermore, we have proved in [An5, Theorems 3.2 and 4.1] that for integers
$q$ and $i$ with $q\geqq i^{2}\geqq1$,%
\begin{equation}
\dim\Omega_{4q}^{h-eq}/(\Omega_{4q}^{h-eq}\cap\mathrm{Ker}(\mathcal{C}%
_{2i}))\otimes\mathbb{Q}=\dim H_{4q-4i^{2}}(BSO;\mathbb{Q}).
\end{equation}

The following theorem follows from (5.1), Proposition 4.2 and Corollary 4.4.

\begin{theorem}
Let $\ell$, $q$ and $i$ be integers with $\ell\geqq0$\ and $q\geqq i^{2}$. Let
$k\geqq4q+\ell+1$.\ There exists a cobordism class $[f:N\rightarrow
P]\in\Omega_{4q}^{h-eq}$ such that $\mathbf{c}_{\mathbb{Z}}(\Sigma^{2i}%
,\tau_{N}-f^{\ast}(\tau_{P}))$ is not a torsion element and that if
$4i^{3}-2i^{2}\geqq4q+\ell\geqq4i^{2}+\ell$, then $f$\ is not cobordant in
$\Omega_{4q}^{h-eq}$ to any $\mathcal{O}_{\ell}^{k}$-regular map.
\end{theorem}

We can prove the following theorem using Theorem 5.1 by applying the same
argument in the proof of [An5, Theorem 0.2]. However, Theorem 1.2 is very
important in the following and the situation is rather different. Therefore,
we give its proof.

\begin{theorem}
Let $\ell$, $q$ and $i$ be given integers with $\ell\geqq0$\ and $q\geqq
i^{2}$. Let $k\geqq8q+\ell+1$.\ If $4i^{3}-2i^{2}\geqq4q+\ell\geqq4i^{2}+\ell
$, then there exists a closed connected oriented $8q$-manifold $P$ and a
homotopy equivalence $f:P\rightarrow P$\ of degree $1$\ such that
$\mathbf{c}_{\mathbb{Z}}(\Sigma^{2i},\tau_{P}-f^{\ast}(\tau_{P}))\neq0$ and
that $f$ is not cobordant in $\Omega_{8q}^{h-eq}$ to any $\mathcal{O}_{\ell
}^{k}$-regular homotopy equivalence of degree $1$.
\end{theorem}

\begin{proof}
It follows from Theorem 5.1 that there exists a homotopy equivalence
$f:N\rightarrow P$\ of degree $1$ between $4q$-manifolds such that
$\mathbf{c}_{\mathbb{Z}}(\Sigma^{2i},\tau_{N}-f^{\ast}(\tau_{P}))$\ is not a
torsion element. Let $f^{-1}:P\rightarrow N$\ be a homotopy inverse of $f$.
Define $g:N\times P\rightarrow N\times P$ by $g(x,y)=(f^{-1}(y),f(x))$. We
have $k\geq\dim N\times P+\ell+1$. If we prove\ that $\mathbf{c_{\mathbb{Z}}%
}(\Sigma^{2i},\tau_{N\times P}-g^{\ast}(\tau_{N\times P}))$ does not vanish,
then, by Corollary 4.4, $g$\ is not\ homotopic to any $\mathcal{O}_{\ell}^{k}%
$-regular map. We set $\xi=\tau_{N\times P}-g^{\ast}(\tau_{N\times P}%
)=\tau_{N}\times\tau_{P}-f^{\ast}(\tau_{P})\times(f^{-1})^{\ast}(\tau_{N})$.
Then%
\begin{align*}
p_{j}(\xi)  &  =\sum_{s+t=j}p_{s}(\tau_{N}\times\tau_{P})\overline{p}%
_{t}(f^{\ast}(\tau_{P})\times(f^{-1})^{\ast}(\tau_{N}))\\
&  =\sum_{s+t=j}\sum_{%
\begin{array}
[c]{l}%
s_{1}+s_{2}=s\\
t_{1}+t_{2}=t
\end{array}
}p_{s_{1}}(\tau_{N})\overline{p}_{t_{1}}(f^{\ast}(\tau_{P}))\otimes p_{s_{2}%
}(\tau_{P})\overline{p}_{t_{2}}((f^{-1})^{\ast}(\tau_{N}))
\end{align*}
modulo torsion in $H^{\ast}(N;\mathbb{Z})\otimes H^{\ast}(P;\mathbb{Z})$. The
term of $p_{j}(\xi)$\ which lies in $H^{4j}(N;\mathbb{Z})\otimes
H^{0}(P;\mathbb{Z})$\ is equal modulo torsion to%
\[
\sum_{s+t=j}p_{s}(\tau_{N})\overline{p}_{t}(f^{\ast}(\tau_{P}))\otimes
1=p_{j}(\tau_{N}-f^{\ast}(\tau_{P}))\otimes1.
\]
Hence, we have that $\mathbf{c_{\mathbb{Z}}}(\Sigma^{2i},\tau_{N\times
P}-g^{\ast}(\tau_{N\times P})$ is equal to the sum of $\mathbf{c_{\mathbb{Z}}%
}(\Sigma^{2i},\tau_{N}-f^{\ast}(\tau_{P}))\otimes1$ and the other term which
lies in $\Sigma_{j=1}^{i^{2}}H^{4i^{2}-4j}(N;\mathbb{Z})\otimes H^{4j}%
(P;\mathbb{Z})$\ modulo torsion. Since $\mathbf{c_{\mathbb{Z}}}(\Sigma
^{2i},\tau_{N}-f^{\ast}(\tau_{P}))$\ does not vanish, it follows that
$\mathbf{c_{\mathbb{Z}}}(\Sigma^{2i},\tau_{N\times P}-g^{\ast}(\tau_{N\times
P}))$\ does not vanish. This completes the proof.
\end{proof}

We are now ready to prove Theorem 1.3.

\begin{proof}
[Proof of Theorem 1.3]In the proof $k$ refers to a sufficiently large integer.
Let $i_{0}=2$, which is the smallest integer such that $4i^{3}-2i^{2}%
\geqq4i^{2}$ with $q=4$ and $\ell=8$. Then we have, by Theorem 5.2, a closed
connected oriented $8\cdot4$-manifold $P_{0}$ and a homotopy equivalence
$f_{0}:P_{0}\rightarrow P_{0}$\ of degree $1$\ such that $\mathbf{c}%
_{\mathbb{Z}}(\Sigma^{4},\tau_{P_{0}}-f_{0}^{\ast}(\tau_{P_{0}}))\neq0$ and
that $f_{0}$ is not homotopic to any $\mathcal{O}_{8}^{k}$-regular map. By
Remark 3.8 there exists an integer $\ell$ such that $f_{0}$ is homotopic to an
$\mathcal{O}_{\ell}^{k}$-regular map. Let $\ell_{1}$ be such a smallest integer.

We assume the following (A-$t$) for an integer $t\geqq0$, where $\ell_{0}=8$.

(A-$t$) We have constructed integers $\ell_{t}$, $\ell_{t+1}$, $i_{t}$, a
closed oriented $8\cdot i_{t}^{2}$-manifold $P_{t}$ and an $\mathcal{O}%
_{\ell_{t+1}}^{k}$-regular homotopy equivalence $f_{t}:P_{t}\rightarrow P_{t}%
$\ of degree $1$\ such that $4i_{t}^{3}-2i_{t}^{2}\geqq4i_{t}^{2}+\ell_{t}$,
$\ell_{t+1}>\ell_{t}$, $\mathbf{c}_{\mathbb{Z}}(\Sigma^{2i_{t}},\tau_{P_{t}%
}-f_{t}^{\ast}(\tau_{P_{t}}))\neq0$ and that $f_{t}$ is not homotopic to any
$\mathcal{O}_{\ell_{t}}^{k}$-regular map.

Under the assumption (A-$t$) we prove (A-$(t+1)$) with $\ell_{t+1}<\ell_{t+2}%
$. Let $i_{t+1}$ be the smallest integer among the integers $i>0$ with
$4i^{3}-2i^{2}\geqq4i^{2}+\ell_{t+1}$. Then it follows from Theorem 5.2 that
there exist a closed connected oriented $8\cdot i_{t+1}^{2}$-manifold
$P_{t+1}$ and a homotopy equivalence $f_{t+1}:P_{t+1}\rightarrow P_{t+1}$\ of
degree $1$\ such that $\mathbf{c}_{\mathbb{Z}}(\Sigma^{2i_{t}},\tau_{P_{t+1}%
}-f_{t+1}^{\ast}(\tau_{P_{t+1}}))\neq0$ and that $f_{t+1}$ is not homotopic to
any $\mathcal{O}_{\ell_{t+1}}^{k}$-regular map. It follows Remark 3.8 that
there exists an integer $\ell$ such that $f_{t+1}$ is homotopic to an
$\mathcal{O}_{\ell}^{k}$-regular map. Let $\ell_{t+2}$ be the smallest integer
among those integers $\ell$. Hence, we have $\ell_{t+2}>\ell_{t+1}$. This
proves (A-$(t+1)$).

Thus we have defined the sequences $\{i_{t}\}$, $\{ \ell_{t}\}$, closed
connected oriented manifolds $\{P_{t}\}$\ of dimensions $\{8\cdot i_{t}^{2}\}$
and homotopy equivalences $\{f_{t}\}$ of degree $1$ which satisfy the above properties.

Given a positive integer $d$, let%
\begin{align*}
P  &  =P_{0}\times P_{1}\times P_{2}\times\cdots\times P_{d},\\
F_{t}  &  =id_{P_{0}}\times\cdots\times id_{P_{t-1}}\times f_{t}\times
id_{P_{t+1}}\times\cdots\times id_{P_{d}}\text{ \ }(0\leqq t\leqq d),
\end{align*}
and $p=\sum_{t=0}^{d}8\cdot i_{t}^{2}$. We show that $F_{t}\notin
\mathfrak{h}_{\ell_{t}}(P)$ and $F_{t}\in\mathfrak{h}_{\ell_{t+1}}(P)$. Let
$q_{t}:P\rightarrow P_{t}$ be the canonical projection. Then the stable
tangent bundle $\tau_{P}$ is isomorphic to $q_{0}^{\ast}(\tau_{P_{0}})\oplus
q_{1}^{\ast}(\tau_{P_{1}})\oplus\cdots\oplus q_{d}^{\ast}(\tau_{P_{d}})$.
Hence, $\tau_{P}-F_{t}^{\ast}(\tau_{P})$ is equal to%
\begin{align*}
&  q_{0}^{\ast}(\tau_{P_{0}})\oplus q_{1}^{\ast}(\tau_{P_{1}})\oplus
\cdots\oplus q_{d}^{\ast}(\tau_{P_{d}})\\
&  -((q_{0}\circ F_{t})^{\ast}(\tau_{P_{0}})\oplus(q_{1}\circ F_{t})^{\ast
}(\tau_{P_{1}})\oplus\cdots\oplus(q_{d}\circ F_{t})^{\ast}(\tau_{P_{d}}))\\
&  =q_{0}^{\ast}(\tau_{P_{0}})\oplus q_{1}^{\ast}(\tau_{P_{1}})\oplus
\cdots\oplus q_{d}^{\ast}(\tau_{P_{d}})\\
&  -(q_{0}^{\ast}(\tau_{P_{0}})\oplus\cdots\oplus q_{t-1}^{\ast}(\tau
_{P_{t-1}})\oplus(f_{t}\circ q_{t})^{\ast}(\tau_{P_{t}})\oplus\cdots\oplus
q_{d}^{\ast}(\tau_{P_{d}}))\\
&  =q_{t}^{\ast}(\tau_{P_{t}})-(f_{t}\circ q_{t})^{\ast}(\tau_{P_{t}})\\
&  =q_{t}^{\ast}((\tau_{P_{t}})-f_{t}^{\ast}(\tau_{P_{t}})).
\end{align*}
This shows that
\begin{align*}
\mathbf{c}_{\mathbb{Z}}(\Sigma^{2i_{t}},\tau_{P}-F_{t}^{\ast}(\tau_{P}))  &
=\mathbf{c}_{\mathbb{Z}}(\Sigma^{2i_{t}},q_{t}^{\ast}((\tau_{P_{t}}%
)-f_{t}^{\ast}(\tau_{P_{t}}))\\
&  =q_{t}^{\ast}(\mathbf{c}_{\mathbb{Z}}(\Sigma^{2i_{t}},\tau_{P_{t}}%
-f_{t}^{\ast}(\tau_{P_{t}})),
\end{align*}
which does not vanish in $H^{2i_{t}^{2}}(P;\mathbb{Z})$ since $\mathbf{c}%
_{\mathbb{Z}}(\Sigma^{2i_{t}},\tau_{P_{t}}-f_{t}^{\ast}(\tau_{P_{t}}))\neq0$
and since $q_{t}^{\ast}:H^{2i_{t}^{2}}(P_{t};\mathbb{Z})\rightarrow
H^{2i_{t}^{2}}(P;\mathbb{Z})$ is injective. Furthermore, it follows from
Proposition 4.3 that $\Sigma^{2i_{t}}(p,p)\subset W_{\ell+1}^{k}(p,p)$ and
from Corollary 4.4 that $F_{t}$ is not homotopic to any $\mathcal{O}_{\ell
_{t}}^{k}$-regular map. However, since $f_{t}$ is homotopic to an
$\mathcal{O}_{\ell_{t+1}}^{k}$-regular map, $F_{t}$ is also homotopic to an
$\mathcal{O}_{\ell_{t+1}}^{k}$-regular map. This proves the theorem.
\end{proof}

We prepare further results which are necessary to study the filtration in
(1.1). The assertions (i) and (ii) in the following theorem have been proved
in [An2, Theorem 4.8] and [An4, Theorem 4.1] respectively, which are
applications of the relative homotopy principles for $\mathcal{O}$-regular maps.

\begin{theorem}
Let $P$ be orientable and $f:P\rightarrow P$ be a smooth map.

(i) A map $f$ is homotopic to a fold-map if and only if $\tau_{P}$ is
isomorphic to $f^{\ast}(\tau_{P})$.

(ii) If a map $f$ is $\Omega^{1}$-regular, then $f$ is homotopic to an
$\Omega^{(1,1,0)}$-regular map.
\end{theorem}

Let $V(n,p)$ be an algebraic set of $J^{k}(n,p)$\ which is invariant with
respect to the actions of local diffeomorphisms of $(\mathbb{R}^{n},0)$\ and
$(\mathbb{R}^{n},0)$ and Let $V(N,P)$\ be the subbundle of $J^{k}%
(N,P)$\ associated to $V(n,p)$. By [B-H] we have the fundamental class of
$V(N,P)$ under the coefficient group $\mathbb{Z}/2$, and have the Thom
polynomial $\mathbf{c}(V(n,p),\tau_{N}-f^{\ast}(\tau_{P}))$ of $V(N,P)$.

\begin{theorem}
Let $V(p,p)$ be as above. Let $P$ be orientable and $f:P\rightarrow P$ be a
smooth map.

(i) If $f$\ is a homotopy equivalence, then $\mathbf{c}(V(p,p),\tau
_{P}-f^{\ast}(\tau_{P}))$ vanishes.

(ii) $\mathbf{c}_{\mathbb{Z}}(W_{p}^{k}(p,p),\tau_{P}-f^{\ast}(\tau_{P}))=0$
for $p=5,6,7$ and%
\[
\mathbf{c}_{\mathbb{Z}}(W_{8}^{k}(8,8),\tau_{P}-f^{\ast}(\tau_{P}%
))=9P_{2}(\tau_{P}-f^{\ast}(\tau_{P}))+3P_{1}^{2}(\tau_{P}-f^{\ast}(\tau
_{P}))
\]
for $p=8$.

(iii) Let $2\leqq p\leqq8$. Then there exists a section $s$ of $\mathcal{O}%
_{p-1}^{k}(P,P)$ over $P$ with $\pi_{P}^{k}\circ s$ and $f$ being homotopic if
and only if $\mathbf{c}_{\mathbb{Z}}(W_{p}^{k}(p,p),\tau_{P}-f^{\ast}(\tau
_{P}))=0$.
\end{theorem}

\begin{proof}
(i) Let $S(\nu_{P})$ denote the spherical normal fiber space of $P$. It
follows from [Sp] that $S(\nu_{P})$ is equivalent to $f^{\ast}(S(\nu_{P}))$.
Hence, the associated spherical spaces of $\tau_{P}$ and $f^{\ast}(\tau_{P}%
)$\ are equivalent. In particular, the Stiefel-Whitney classes of $\tau
_{P}-f^{\ast}(\tau_{P})$ vanish.

(ii) If $p\leqq8$, then a map $f:P\rightarrow P$ is homotopic to a smooth map
with only $\mathcal{K}$-simple singularities by [MaVI]. According to [F-R],
the integer Thom polynomial of $W_{p}^{k}(p,p)$ is equal to the formula for
$p=8$ and vanish for $p=5,6,7$ in $H^{p}(P;\mathbb{Z)\approx Z}$.

(iii) It follows from the relative homotopy principle for $\mathcal{O}%
_{p-1}^{k}$-regular maps $P\rightarrow P$ that the primary obstruction in
$H^{p}(P;\pi_{p-1}(\mathcal{O}_{p-1}^{k}(p,p))$\ is the unique obstruction for
finding the required section. By an elementary argument we have%
\[
\pi_{p-1}(\mathcal{O}_{p-1}^{k}(p,p))\approx H_{p-1}(\mathcal{O}_{p-1}%
^{k}(p,p);\mathbb{Z})\approx H^{\dim W_{p}^{k}(p,p)}(W_{p}^{k}(p,p);\mathbb{Z}%
).
\]
\ This shows the assertion.
\end{proof}

Finally we study the filtration in (1.1) in the case of $P$ being orientable
and $p\leqq8$ by applying the homotopy principles in Theorems 1.2 and 5.3. We
have $\mathfrak{h}_{p}(P)=\mathfrak{h}(P)$.

\bigskip

\textbf{Examples.}

Case: $p\leqq3$; $\mathfrak{h}_{0}(P)\subset\mathfrak{h}_{1}(P)=\mathfrak{h}%
(P)$.

Since $P$ is parallelizable, $TP$ and $f^{\ast}(TP)$ are trivial. So a map
$f:P\rightarrow P$ is homotopic to a fold-map. We refer the reader to [Ru, 1].

Case: $p=4$; $\mathfrak{h}_{0}(P)\subset\mathfrak{h}_{1}(P)\subset
\mathfrak{h}_{2}(P)=\mathfrak{h}_{3}(P)\subset\mathfrak{h}_{4}(P)$.

It is known that $\mathbf{c}_{\mathbb{Z}}(\Sigma^{4};\tau_{P}-f^{\ast}%
(\tau_{P}))=P_{2}(\tau_{P}-f^{\ast}(\tau_{P}))$. If this class vanish, then
there exists a section $P\rightarrow\Omega^{1}(P,P)$ covering $f$, and hence
an $\Omega^{1}$-regular map by [F]. By Theorems 5.3 and 5.4 we obtain an
$\Omega^{(1,1,0)}$-regular map homotopic to $f$. It has been proved in [Ak]
that $\mathfrak{h}_{0}(P)\neq\mathfrak{h}(P)$ for $P=S^{3}\times S^{1}%
\#S^{2}\times S^{2}$.

Case: $5\leqq p\leqq7$; $\mathfrak{h}_{0}(P)\subset\mathfrak{h}_{1}%
(P)\subset\cdots\subset\mathfrak{h}_{p-1}(P)=\mathfrak{h}_{p}(P)$.

This follows from Theorems 1.2 and 5.4.

Case: $p=8$; $\mathfrak{h}_{0}(P)\subset\mathfrak{h}_{1}(P)\subset
\cdots\subset\mathfrak{h}_{7}(P)\subset\mathfrak{h}_{8}(P)$.

If $9P_{2}(\tau_{P}-f^{\ast}(\tau_{P}))+3P_{1}^{2}(\tau_{P}-f^{\ast}(\tau
_{P}))=0$, then the homotopy class of $f$ lies in $\mathfrak{h}_{7}(P)$ by
Theorems 1.2 and 5.4.

For more precise information we must investigate the obstructions for finding
sections in $\Gamma_{\mathcal{O}_{\ell}^{k}}(P,P)$ related to $W_{\ell+1}%
^{k}(p,p)$.

Department of Mathematical Sciences

Faculty of Science, Yamaguchi University

Yamaguchi 753-8512, Japan

E-mail: andoy@yamaguchi-u.ac.jp

\begin{thebibliography}{99999}                                                                                            %


\bibitem[Ak]{}S. Akbulut, Scharlemann's manifolds is standard, Ann. of Math.
149(1999), 497-510.

\bibitem[An1]{}Y. Ando, Elimination of Thom-Boardman singularities of order
two, J. Math. Soc. Japan 37(1985), 471-487.

\bibitem[An2]{}Y. Ando, Fold-maps and the space of base point preserving maps
of spheres, J. Math. Kyoto Univ. 41(2002), 691-735.

\bibitem[An3]{}Y. Ando, Existence theorems of fold-maps, Japanese J. Math.
30(2004), 29-73.

\bibitem[An4]{}Y, Ando, Stable homotopy groups of spheres and higher
singularities, J. Math. Kyoto Univ. 46(2006), 147-165.

\bibitem[An5]{}Y. Ando, Nonexistence of homotopy equivalences which are
$C^{\infty}$ stable or of finite codimension, Topol. Appl. 153(2006), 2962-2970.

\bibitem[An6]{}Y. Ando, A homotopy principle for maps with prescribed
Thom-Boardman singularities, Trans. Amer. Math. Soc. 359(2007), 489-515.

\bibitem[An7]{}Y. Ando, The homotopy principle for maps with singularities of
given $\mathcal{K}$-invariant class, J. Math. Soc. Japan 59(2007), 557-582.

\bibitem[Bo]{}J. M. Boardman, Singularities of differentiable maps, IHES Publ.
Math. 33(1967), 21-57.

\bibitem[B-H]{}A. Borel and A. Haefliger, La classe d'homologie fundamental
d'un espace analytique, Bull. Soc. Math. France, 89(1961), 461-513.

\bibitem[Br]{}W. Browder, Surgery on Simply-connected Manifolds,
Springer-Verlag, Berlin Heiderberg, 1972.

\bibitem[duP1]{}A. du Plessis, Maps without certain singularities, Comment.
Math. Helv. 50(1975), 363-382.

\bibitem[duP2]{}A. du Plessis, Homotopy classification of regular sections,
Compos. Math. 32(1976), 301-333.

\bibitem[duP3]{}A. du Plessis, Contact invariant regularity conditions,
Springer Lecture Notes 535(1976), 205-236.

\bibitem[duP4]{}A. du Plessis, On mappings of finite codimension, Proc. London
Math. Soc. 50(1985), 114-130.

\bibitem[E1]{}Ja. M. $\grave{\mathrm{E}}\mathrm{lia}\check{\mathrm{s}%
}\mathrm{berg}$, On singularities of folding type, Math. USSR. Izv. 4(1970), 1119-1134.

\bibitem[E2]{}Ja. M. $\grave{\mathrm{E}}\mathrm{lia}\check{\mathrm{s}%
}\mathrm{berg}$, Surgery of singularities of smooth mappings, Math. USSR. Izv.
6(1972), 1302-1326.

\bibitem[F]{}S. Feit, $k$-mersions of manifolds, Acta Math. 122(1969), 173-195.

\bibitem[F-R]{}L. Feh\'{e}r and R. Rim\'{a}nyi, Thom polynomials with integer
coefficients, Illinois J. Math. 46(2002), 1145-1158.

\bibitem[G1]{}M. Gromov, Stable mappings of foliations into manifolds, Math.
USSR. Izv. 3(1969), 671-694.

\bibitem[G2]{}M. Gromov, Partial Differential Relations, Springer-Verlag,
Berlin, Heidelberg, 1986.

\bibitem[H]{}M. Hirsch, Immersions of manifolds, Trans. Amer. Math. Soc.
93(1959), 242-276.

\bibitem[I-K]{}S. Izumiya and Y. Kogo, Smooth mappings of bounded
codimensions, J. London Math. Soc. 26(1982), 567-576.

\bibitem[K-M]{}M. A. Kervaire and J. W. Milnor, Groups of homotopy spheres: I,
Ann. Math. 77(1963), 504-537.

\bibitem[L]{}H. I. Levine, Singularities of differentiable maps, Proc.
Liverpool Singularities Symposium, I, Springer Lecture Notes in Math. Vol.
192, 1-85, Springer-Verlag, Berlin, 1971.

\bibitem[M-M]{}I. Madsen and R. J. Milgram, The Classifying Spaces for Surgery
and Cobordism of Manifolds, Ann. Math. Studies 92, Princeton Univ. Press,
Princeton, 1979.

\bibitem[Mart]{}J. Martinet, D\'{e}ploiements versels des applications
diff\'{e}rentiables et classification des applications stables, Springer
Lecture Notes in Math. Vol. 535, 1-44, Spribger-Verlag, Berlin, 1976.

\bibitem[MaIII]{}J. N. Mather, Stability of $C^{\infty}$ mappings, III:
Finitely determined map-germs, Publ. Math. Inst. Hautes \'{E}tud. Sci.
35(1968), 127-156.

\bibitem[MaIV]{}J. N. Mather, Stability of $C^{\infty}$ mappings, IV:
Classification of stable germs by $\mathbb{R}$-algebra, Publ. Math. Inst.
Hautes \'{E}tud. Sci. 37(1970), 223-248.

\bibitem[MaV]{}J. N. Mather, Stability of $C^{\infty}$ mappings: V,
Transversality, Adv. Math. 4(1970), 301-336.

\bibitem[MaTB]{}J. N. Mather, On Thom-Boardman singularities, Dynamical
Systems, Academic Press, 1973, 233-248.

\bibitem[O]{}T. Ohmoto, Vassiliev complex for contact classes of real smooth
map-germs, Res. Fac. Sci. Kagoshima Univ. 27(1994), 1-12.

\bibitem[Ph]{}A. Phillips, Submersions of open manifolds, Topology 6(1967), 171-206.

\bibitem[Po]{}I. R. Porteous, Simple singularities of maps, Proc. Liverpool
Singularities Symp. I, Springer Lecture Notes in Math. 192(1971), 286-307.

\bibitem[Ro]{}F. Ronga, Le calcul de la classe de cohomologie enti\`{e}re dual
a $\Sigma^{k}$, Proc. Liverpool Singularities Symp. I, Springer Lecture Notes
in Math. 192(1971), 313-315.

\bibitem[Ru]{}J. W. Rutter, Homotopy self-equivalences 1988-1999, Contemporary
Math. 274(2001), 1-11.

\bibitem[Sa]{}O. Saeki, Fold maps on 4-manifolds, Comment. Math. Helv.,
78(2003), 627-647.

\bibitem[Sm]{}S. Smale, The classification of immersions of spheres in
Euclidean spaces, Ann. Math. 327-344, 69(1969).

\bibitem[Sp]{Sp}M. Spivak, Spaces satisfying Poincar\'{e} duality, Topology
6(1969), 77-102.

\bibitem[St]{}N. Steenrod, The Topology of Fibre Bundles, Princeton Univ.
Press, Princeton, 1951.

\bibitem[Su]{}D. Sullivan, Triangulating homotopy equivalences, Thesis,
Princeton Univ., 1965.

\bibitem[T]{}R. Thom, Les singularit\'{e}s des applications
diff\'{e}rentiables, Ann. Inst. Fourier 6(1955-56), 43-87.

\bibitem[W]{}C. T. C. Wall, Finite determinacy of smooth map germs, Bull.
London Math. Soc. 13(1981), 481-539.
\end{thebibliography}
\end{document}